\documentclass[13pt]{article}  % list options between brackets
\usepackage{amssymb}              % list packages between braces
\usepackage{amsthm}
\usepackage{eucal}
\usepackage{verbatim}
\usepackage[dvips]{graphicx}
\usepackage{multirow}
\usepackage{fancyhdr}
\usepackage{color}
\usepackage{enumerate}
\usepackage{calrsfs}
\usepackage{fullpage}
\usepackage{amsmath}
\usepackage[pagewise]{lineno}%\linenumbers

\newtheorem{theorem}{Theorem}
\newtheorem{lemma}{Lemma}

\makeatletter
\newcommand{\leqnomode}{\tagsleft@true}
\newcommand{\reqnomode}{\tagsleft@false}
\makeatother
\def\({\begin{eqnarray}}
\def\){\end{eqnarray}}
\def\[{\begin{eqnarray*}}
\def\]{\end{eqnarray*}}
\def\part#1#2{\frac{\partial #1}{\partial #2}}

\def\R{\mathbb{R}}
\def\d{\mathrm{d}}
\def\tot#1#2{\frac{\d #1}{\d #2}}
\def\eps{\varepsilon}

\def\M{\mathbb{M}}

\begin{document}

\title{Exponential decay for negative feedback loop with distributed delay}   % type title between braces
\author{Jan Haskovec\footnote{Computer, Electrical and Mathematical Sciences \& Engineering, King Abdullah University of Science and Technology, 23955 Thuwal, KSA.
jan.haskovec@kaust.edu.sa}}

\date{}

\maketitle

\begin{abstract}
We derive sufficient conditions for exponential decay of solutions
of the {delay negative feedback equation} with distributed delay.
The conditions are written in terms of exponential moments of the distribution.
Our method only uses elementary tools of calculus and is robust towards
possible extensions to more complex settings, in particular,
systems of delay differential equations.
We illustrate the applicability of the method to particular distributions - 
Dirac delta, Gamma distribution, uniform and truncated normal distributions.
%deriving explicit conditions for exponential decay of the corresponding solutions.
\end{abstract}
\vspace{2mm}

\textbf{Keywords}: Negative feedback loop, distributed delay, exponential decay.
\vspace{2mm}

\textbf{2010 MR Subject Classification}: 34K06, 34K25, 34K11.
\vspace{2mm}

%%%%%%%%%%%%%%%%%%%%%%%%%%%%%%%%%%%%%%%%
\section{Introduction and main result}\label{sec:Intro}
In this paper we derive sufficient conditions for exponential decay of solutions
of the {delay negative feedback equation} with distributed delay,
\( \label{eq:main}
   \dot u(t) = - F_P[u](t) := - \int_0^{\infty} u(t-s) \d {P}(s) \qquad\mbox{for } t>0,
\)
where $P$ is a probability measure on $[0,\infty)$.
Note that the normalization $\int_0^\infty \d P=1$ can be imposed
by an eventual rescaling of the time variable.
For simplicity, we consider \eqref{eq:main} subject to the constant initial
datum $u(s) \equiv 1$ for $s\leq 0$; alternatively, we may assume that
\eqref{eq:main} holds globally, i.e., for all $t\in\R$.

The importance of equation \eqref{eq:main}, also called a linear retarded functional differential equation,
stems from the fact that it can be seen as a linearization of many nonlinear models in biology and physics involving delay.
As such, it has been a long-standing subject of interest of the mathematical community.
Basic theory for delay differential equations and functional differential equations can be
found in, e.g., \cite{Bellman} and \cite{Hale}, while \cite{Erneux} and \cite{Smith} focus on applications.
The theory typically focuses on two qualitative aspects of delay/functional differential equations -
(asymptotic) stability of the steady state solutions \cite{Bernard-Crauste, Krizstin},
and oscillatory behavior \cite{Agarwal1, Agarwal2, Berezansky, Ladas-Ladde}.
This note aims to contribute to the study of the latter aspect
by deriving sufficient conditions for the solution of \eqref{eq:main} %--\eqref{IC}
to decay monotonically (exponentially) to zero.
In contrast to the traditional approach, based on studying the characteristic
equation, our method only uses elementary tools of calculus.
It provides relatively simple sufficient conditions for exponential decay of the solution,
written in terms of the exponential moments of the distribution $P$.
Due to its simplicity, it can be applied to systems of delay differential equations,
where the analysis of the characteristic equation would be prohibitively complex;
see \cite{HasMar} for a recent application.
Let us note that in the case when $P$ is a Dirac measure,
a slight modification of the method leads to an optimal (i.e., equivalent) condition
for monotone decay of the solution.
%, despite this simplicity, our sufficient conditions are not too far from being optimal
%(at least) for the Dirac and exponential distribution of the delay.
%(see Sections \ref{sec:ex} and \ref{sec:num}).

In the sequel we shall denote, for $\mu>0$, the exponential moment of $p$ by
\(   \label{M_p}
   \M_P(\mu) := \int_0^\infty e^{\mu s} \d P(s).
\)

\begin{theorem}\label{thm:main}
If there exists some $\mu>1$ such that
\(  \label{ass:mu1}
    \M_P(2\mu) \leq \mu^2
\)
and
\(  \label{ass:mu2}
   {\M_P(\mu)} (\M_P(\mu) - 1) < \mu,
\)
then the solution $u=u(t)$ of \eqref{eq:main} %--\eqref{IC}
converges monotonically exponentially to zero as $t\to\infty$
with rate at least
\[
   2 \left( \frac{\M_P(\mu)\bigl(\M_P(\mu) - 1 \bigr)}{\mu} - 1 \right) .
\]
\end{theorem}

Let us note that in the standard theory \eqref{eq:main} is called \emph{nonoscillatory}
if there exists an initial datum $u_0$ such that the solution of the initial value problem %\eqref{eq:main}, \eqref{IC}
is eventually positive or eventually negative (see Definition 1.1 in \cite{Agarwal1}).
%Therefore, considering the constant initial datum which reduces conditions \eqref{IC_cond1}, \eqref{IC_cond2}
%to $K=1$,
Therefore, Theorem \ref{thm:main} provides sufficient conditions for \eqref{eq:main} to be nonoscillatory.
Let us again point out the relative simplicity of the conditions \eqref{ass:mu1}, \eqref{ass:mu2}, being
only written in terms of the exponential moments of the distribution $P$.

The proof of Theorem \ref{thm:main}
% uses only elementary tools of calculus and is based on a forward-backward estimate for $y(t):=u^2(t)/2$
%and a decay estimate for the same quantity. It is carried out in Section \ref{sec:Proof}.
is based on suitable decay estimates for the quantity $y(t):=u^2(t)/2$ and is carried out in Section \ref{sec:Proof}.
In Section \ref{sec:ex} we show the applicability of the result to particular choices of the measure $P$.
%where it leads to explicit sufficient conditions for monotone decay.
First, we consider the Dirac measure concentrated at $\tau>0$, $P(s) = \delta(s - \tau)$,
which turns \eqref{eq:main} into the simple negative feedback equation with constant delay $\tau>0$,
\(   \label{eq:constdelay}
   \dot u(t) = - u(t-\tau)  \qquad\mbox{for } t>0.
\)
We shall show that the conditions \eqref{ass:mu1} and \eqref{ass:mu2} are satisfied if
\[
   \tau < \ln\sqrt{2}.
\]
Moreover, we shall show that by a slight modification of the proof of Theorem \ref{thm:main}
we obtain monotone decay of the solution as soon as $\tau \leq e^{-1}$.
This result is sharp since it is known that for $\tau > e^{-1}$ the nontrivial solutions of \eqref{eq:constdelay}
must oscillate \cite{Gyori-Ladas}.
The second example is the Gamma distribution $\d P(s) = \lambda^k s^{k-1} e^{-\lambda s}/\Gamma(k)$
with shape parameter $k >0$ and rate parameter $\lambda>0$.
Here we derive explicit sufficient conditions for satisfiability of \eqref{ass:mu1}, \eqref{ass:mu2}.
In the special case of $k=1$, which corresponds to the exponential distribution,
we show that the solution is nonoscillatory if $\lambda \geq 3^\frac{3}{2} \approx 5.196$.
The optimal condition for nonoscillation is $\lambda \geq 4$, see \cite{Berezansky-Braverman}.
Finally, for the uniform and truncated normal distributions we resolve the conditions
\eqref{ass:mu1}, \eqref{ass:mu2} numerically.

%%%%%%%%%%%%%%%%%%%%%%%%%%%%%%%%%%%%%%%%
\section{Proof of the main result}\label{sec:Proof}

In this section we assume that $u=u(t)$ is a solution of \eqref{eq:main} %--\eqref{IC}
subject to the constant initial datum, and we introduce the notation
\(
   y(t) &:=& u^2(t)/2 \qquad\mbox{for } t\geq 0,\\
       &:=& u_0^2(t)/2 \qquad\mbox{for } t< 0.     \nonumber
\)

\begin{lemma} \label{lem:delayEst}
%Let $u=u(t)$ be a solution of \eqref{eq:main}--\eqref{IC} and let $y(t):=u^2(t)/2$.
%If there exists some $\mu>0$ such that
%\(  \label{ass:mu1}
%    \M_P(2\mu) < \mu^2,
%\)
%with $\M_P(\mu)$ defined in \eqref{M_p},
If assumption \eqref{ass:mu1} is verified for some $\mu>1$,
then for all $t>0$ and $s>0$,
\(  \label{fb}
   e^{-2\mu s} y(t) < y(t-s) < e^{2\mu s} y(t).
 \)
\end{lemma}

\begin{proof}
%According to assumption \eqref{IC_cond2} we have
We have
\[
   \left| \frac{\dot y(0+)}{y(0)} \right| = 2 \left| \frac{\dot u(0+)}{u(0)} \right|
      = 2 \left| \frac{F_P[u](0)}{u(0)} \right| = 2 < 2\mu.
\]
%The continuity of $u_0$ on $(-\infty, 0]$ implies continuity of $\dot y(t)$ for all $t>0$.
Due to the continuity of $\dot y(t)$ for $t>0$, there exists $T>0$ such that
\(  \label{proof:y1}
   \left| \frac{\dot{y}(t)}{y(t)} \right| < 2\mu \qquad\mbox{for } t < T.
\)
%for $t\in (0,T)$.
%With assumption \eqref{IC_cond1}, the validity of \eqref{proof:y1}  is extended to all $t<T$.
%\[ \left| \frac{\dot y(s)}{y(s)} \right| = 2 \left| \frac{\dot u_0(s)}{u_0(s)} \right| \leq \mu \]

We claim that \eqref{proof:y1} holds for all $t\in\R$, i.e., $T=+\infty$.
For contradiction, assume that $T<+\infty$, then again by continuity we have
\( \label{proof:y2}
   |\dot{y}(T)| = 2\mu y(T).
\)
Integrating \eqref{proof:y1} on the time interval $(T-s,T)$ with $s>0$ yields
\(  \label{est_temp}
   %e^{-\mu s} y(T) <
      y(T-s) < e^{2\mu s} y(T).
\)
Consequently,
\(  \label{proof:y3}
   F_P[y](T) = \int_0^\infty y(T-s) \d P(s) < y(T) \int_0^\infty e^{2\mu s} \d P(s) = y(T) \M_P(2\mu).
\)
Using the Young inequality with some $\eps>0$, we have
\[
   |\dot y(T)| = |u(T) F_P[u](T)| %\int_0^\infty u(t-s) \d P(s)
      \leq \frac{\eps}{2} u(T)^2 + \frac{1}{2\eps} \bigl( F_P[u](T) \bigr)^2
      %&\leq& \eps u(T)^2 + \frac{1}{\eps} \bigl( F_P[u](T) \bigr)^2 \\
%      &\leq& \eps y(T) + \frac{1}{\eps} F_P[y](T) \\
%    &<& y(T) \left( \eps + \frac{1}{\eps} \M_P(2\mu) \right).
\]
and with Jensen inequality
\[
   \frac{1}{2\eps} \bigl( F_P[u](T) \bigr)^2 =
       \frac{1}{2\eps} \left( \int_0^\infty u(T-s) \d P(s) \right)^2 \leq
        \frac{1}{2\eps}  \int_0^\infty u(T-s)^2 \d P(s)
        = \frac{1}{\eps} F_P[y](T).
\]
Consequently, with \eqref{proof:y3} we arrive at
\[
    |\dot y(T)| \leq  y(T) \left( \eps + \frac{1}{\eps} \M_P(2\mu) \right)
\]
Optimization in $\eps>0$ gives $\eps:= \sqrt{\M_P(2\mu)}$, so that we have
\[
   |\dot y(T)| < 2y(T) \sqrt{\M_P(2\mu)},
\]
and, with assumption \eqref{ass:mu1} we finally arrive at
\[
   |\dot y(T)| < 2 \mu y(T),
\]
a contradiction to \eqref{proof:y2}.
Consequently, \eqref{proof:y1} holds with $T:=\infty$,
and an integration on the interval $(t-s,t)$ implies \eqref{fb}.
\end{proof}

\begin{lemma} \label{lem:decay}
If assumptions \eqref{ass:mu1} and \eqref{ass:mu2} are verified for some $\mu>1$,
then we have, along the solutions of \eqref{eq:main},
\(   \label{eq:L2}
   \dot y(t) < 2\left( \frac{\M_P(\mu)\bigl(\M_P(\mu) - 1 \bigr)}{\mu} - 1 \right) y(t)
\)
for all $t>0$.
\end{lemma}

\begin{proof}
For $t>0$ we have
\(  \label{est_dot_y}
   \dot y = - u F_P[u] = (u-F_P[u]) u - u^2 \leq |u-F_P[u]| |u| - u^2,
\)
and
\(   \label{u-Fu}
   |u(t)-F_P[u](t)| \leq \int_0^\infty |u(t) - u(t-s)| \d P(s).
\)
With \eqref{eq:main} we have
\[
   %|u(t) - u(t-s)| \leq \int_{\max\{0,t-s\}}^t |\dot u(\sigma)| \d\sigma + ...
   |u(t) - u(t-s)| \leq \int_{t-s}^t |\dot u(\sigma)| \d\sigma \leq \int_{t-s}^t |F_P[u](\sigma)| \d\sigma,
\]
%\textbf{[here problem when $t-s<0$!]}
and with Lemma \ref{lem:delayEst},
\[
   |F_P[u](\sigma)| \leq \int_0^\infty |u(\sigma-\theta)| \d P(\theta)
     < |u(\sigma)| \int_0^\infty e^{\mu\theta} \d P(\theta)
     = |u(\sigma)| \M_P(\mu).
\]
Using Lemma \ref{lem:delayEst} %\eqref{fb}
again, we obtain
\[
   \int_{t-s}^t |F_P[u](\sigma)| \d\sigma &<&  \M_P(\mu) \int_{t-s}^t |u(\sigma)| \d\sigma \\
   &<& \M_P(\mu) |u(t)| \int_0^s e^{\mu \theta} \d\theta
   = \frac{\M_P(\mu)}{\mu} \left(e^{\mu s} - 1 \right) |u(t)|,
\]
and inserting into \eqref{u-Fu},
\[
   |u(t)-F_P[u](t)| < \frac{\M_P(\mu)\bigl(\M_P(\mu) - 1 \bigr)}{\mu} |u(t)|.
\]
Using this in \eqref{est_dot_y} gives
\[
   \dot y(t) < \left( \frac{\M_P(\mu)\bigl(\M_P(\mu) - 1 \bigr)}{\mu} - 1 \right) u^2(t)
\]
and \eqref{eq:L2} follows.
\end{proof}

The statement of Theorem \ref{thm:main} follows directly from the above two Lemmata.
Let us note that the above proofs apply without modification also to the case
of global solutions, where \eqref{eq:main} holds on the whole real line.

%%%%%%%%%%%%%%%%%%%%%%%%%%%%%%%%%
\section{Application to generic distributions} \label{sec:ex}
We show the application of Theorem \ref{thm:main} to the Dirac delta and
exponential distribution, where it provides explicit conditions for monotone
decay of the solution.

%%%%%%%%%%%%%%%%%%%%%%%%%%%%%%%%%%
\subsection{Dirac delta.}\label{subsec:Dirac}
We choose $P(s) = \delta(s-\tau)$ for a fixed $\tau>0$.
Then $F_P[u](t) = u(t-\tau)$ and \eqref{eq:main} transforms to the negative feedback loop with constant delay,
\(   \label{eq:tau}
     \dot u(t) = - u(t-\tau).
\)
Delay negative feedback is arguably the simplest nontrivial delay differential equation.
Despite its simplicity, it exhibits a surprisingly rich qualitative dynamics, depending on the value of $\tau >0$.
An analysis of the corresponding characteristic equation
\[
   z + \tau e^{-z} = 0,
\]
where $z\in\mathbb{C}$, reveals that:
\begin{itemize}
\item
If $0 < \tau \leq e^{-1}$, then $u=0$ is asymptotically stable.
\item
If $e^{-1} < \tau < \pi/2$, then $u=0$ is asymptotically stable,
but every nontrivial solution of \eqref{eq:tau} is oscillatory.
\item
If $\tau=\pi/2$, then periodic solutions exist.
\item
If $\tau > \pi/2$, then $u=0$ is unstable.
\end{itemize}
In fact, if $\tau \leq e^{-1}$, the solutions subject to the constant initial datum
do not oscillate and tend monotonically to zero as $t\to\infty$.
If $\tau$ becomes larger than $e^{-1}$ but
smaller than $\pi/2$, the nontrivial solutions must oscillate
(i.e., change sign infinitely many times as $t\to\infty$),
but the oscillations are damped and vanish as $t\to\infty$.
Finally, for $\tau > \pi/2$ the nontrivial solutions oscillate with unbounded amplitude
as $t\to\infty$.
We refer to Chapter 2 of \cite{Smith} and \cite{Gyori-Ladas} for details.

With $P(s) = \delta(s-\tau)$ we readily have $\M_P[\mu] = e^{\mu\tau}$ and condition \eqref{ass:mu1} reads
\[
   e^{\mu\tau} \leq \mu.
\]
A simple analysis reveals that this condition is satisfiable if and only if $\tau\leq e^{-1}$ and
that for $\tau=e^{-1}$ we have to choose $\mu=e$.
However, condition \eqref{ass:mu2} is more restrictive, since for $\tau=e^{-1}$, $\mu=e$
it reads $e(e-1)<e$ and clearly is not satisfied.
Since the left-hand side of \eqref{ass:mu2} grows exponentially in $\mu\tau$,
we are motivated to pick $\mu$ as the solution of $\mu=e^{\mu\tau}$.
Inserting into \eqref{ass:mu2} gives then
\[
   \mu \left(\mu -1 \right) < \mu,
\]
i.e., $\mu<2$. Going back to \eqref{ass:mu1} with $\mu=2$, we obtain
the critical value of $\tau=\ln\sqrt{2} \approx 0.3466$.
This is less optimal than the condition $\tau<e^{-1}\approx 0.3679$ by factor of approx. $0.94$.
However, let us note that with a slight modification of the proof of Lemma \ref{lem:decay} we
can obtain monotone decay of the solution to zero
for all $\tau \leq e^{-1}$.
%Indeed, denoting $\widetilde u:=u(t-\tau)$, we have for $t>0$
Indeed, we replace \eqref{est_dot_y} by %have for $t>0$
\(  \label{for difference}
   \dot y = - uF_P[u] %= -\lambda (u-F_P[u])F_P[u] - \lambda F_P[u]^2
     \leq |u-F_P[u]| |F_P[u]| - F_P[u]^2,
\)
and, restricting to $t>\tau$,
\(   \label{u-wtu}
   |u(t)-F_P[u](t)| = |u(t)-u(t-\tau)| &\leq& \int_{t-\tau}^t |\dot u(s)| \d s \\
      &\leq& \int_{t-\tau}^t |u(s-\tau)| \d s = \int_{-\tau}^0 |u(t+s-\tau)| \d s.
   %\sqrt{2} \int_{-\tau}^0 \sqrt{y(t-\tau+s)} \d s,
\)
%where we used the definition $y = u^2/2$.
Using Lemma \ref{lem:delayEst} with $\mu:=e$ we obtain $|u(t-\tau+s)| < e^{-e s} |u(t-\tau)|$ for $s\in(-\tau,0)$, which gives
\[
   |u(t)-F_P[u](t)| < |u(t-\tau)| \int_{0}^\tau e^{e s} \d s
    %  \leq \tau \sqrt{2F_P[y]} e^{e\tau}.
    = \frac{e^{e\tau}-1}{e} |u(t-\tau)|.
\]
%where we denoted $\widetilde y:=y(t-\tau)$.
Therefore, recalling that $F_P[u](t) = u(t-\tau)$, we have for $t>\tau$,
\(   \label{almost there}
   \dot y < \left( \frac{e^{e\tau}-1}{e} - 1 \right) F_P[u]^2.
\)
Consequently, for $\tau \leq e^{-1}$ we readily have $\dot y < - e^{-1}F_P[u]^2 \leq 0$ for all $t>\tau$,
and we conclude that $y=y(t)$, and thus $u=u(t)$, tend monotonically to zero as $t\to\infty$.
%so that $y=y(t)$ has a nonnegative limit as $t\to \infty$.
%If this limit was strictly positive, then the limit of the right-hand side in \eqref{almost there}
%would be strictly negative, which would imply $\limsup_{t\to\infty} \dot y(t) < 0$, a contradiction.
%We conclude that if $\tau < e^{-1}$, then $u=u(t)$ tends monotonically to zero as $t\to\infty$.
Let us point out that this result is sharp since we know that if $\tau > e^{-1}$,
the solution must oscillate \cite{Gyori-Ladas}.

%%%%%%%%%%%%%%%%%%%%%%%%%%%%%%%%%%%%%
\subsection{Gamma distribution.}\label{subsec:Gamma}
For the Gamma distribution $\d P(s) = \lambda^k s^{k-1} e^{-\lambda s}/\Gamma(k)$
with shape parameter $k >0$ and rate parameter $\lambda>0$ we have for $\mu<\lambda$,
\[
   \M_P[\mu] = \left( \frac{\lambda}{\lambda-\mu} \right)^k.
\]
Condition \eqref{ass:mu1} reads
\[
   \left( \frac{\lambda}{\lambda-2\mu} \right)^k \leq \mu^2
\]
and is satisfiable if and only if
\(   \label{k-cond}
   \lambda \geq \frac{(k+2)^\frac{k+2}{2}}{k^\frac{k}{2}}.
\)
with $\mu := \frac{\lambda}{k+2}$.
%The value of $\mu$ corresponding to equality in the above condition is $\mu = \left( \frac{k+2}{k} \right)^\frac{k}{2}$.
Inserting this value into condition \eqref{ass:mu2}, we obtain
\[
   \frac{[k(k+2)]^\frac{k}{2}}{(k+1)^k} \left( \left(\frac{k+2}{k+1}\right)^k -1 \right) < 1,
\]
and an inspection reveals that this is satisfied for all $k\leq 4$.
Consequently, \eqref{eq:main} is nonoscillatory (at least) for $k\leq 4$ if $\lambda$ satisfies \eqref{k-cond}.

Let us point out the special case of $k=1$, which corresponds to the exponential distribution.
Condition \eqref{k-cond} gives here $\lambda \geq 3^\frac{3}{2} \approx 5.196$,
while the optimal condition for nonoscillation is $\lambda \geq 4$,
see \cite{Berezansky-Braverman}.

%%%%%%%%%%%%%%%%%%%%%%%%%%%%%%%%%%
\subsection{Uniform distribution}\label{subsec:Uniform}
For $\d P(s) = \frac{1}{b-a}\chi_{[a,b]}(s) \d s$ with $0\leq a<b$ we have
\[
   \M_P[\mu] = \frac{ e^{\mu b} - e^{\mu a}}{(b-a)\mu}.
\]
Combining the rough estimate
\[
   e^{\mu a} \leq \M_P[\mu] \leq e^{\mu b}
\]
with the results of Section \ref{subsec:Dirac}
implies, as expected, that the solution is nonoscillatory whenever $b<\ln\sqrt{2}$.
On the other hand, conditions \eqref{ass:mu1}, \eqref{ass:mu2} cannot be satisfied if $a>\ln\sqrt{2}\approx 0.3466$.
We resolved the conditions \eqref{ass:mu1}, \eqref{ass:mu2} numerically,
using the matlab routine {\tt fminbnd}. The resulting critical curve is plotted
in Fig. \ref{fig:uniform}. The upper limit on $a$ is approx. $0.346$, in agreement with the analytical result.
On the other hand, for values of $a$ close to zero, the interval length $b-a$ can up to approx. $0.59$.

\begin{figure}
\centerline{
%\hskip 3mm
\includegraphics[width=0.6\columnwidth]{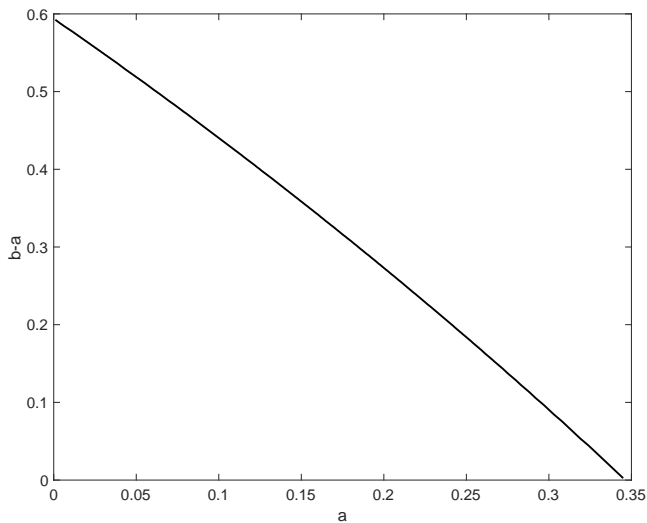}}
\caption{Numerical calculation of the critical value of the interval length $b-a$
as a function of the parameter $a$ for the uniform distribution on $[a,b]$.
}
\label{fig:uniform}
\end{figure}

%%%%%%%%%%%%%%%%%%%%%%%%%%%%%%%%%%%%%
\subsection{Truncated Gaussian distribution.}\label{subsec:truncGauss}
For the truncated normal distribution on $(0,\infty)$ with parameters $m\in\R$ and $\sigma>0$ we have
\(  \label{truncN}
   \d P(s) = \frac{1}{\sqrt{2\pi\sigma^2}} \frac{\exp\left(-\frac12 \left(\frac{s-m}{\sigma}\right)^2 \right)}{\Phi\left(\frac{m}{\sigma}\right)} \d s,
\)
and
\[
   \M_P[\mu] = \frac{\Phi\left(\frac{m}{\sigma}+\sigma \mu\right)}{\Phi\left(\frac{m}{\sigma}\right)}  \exp\left({m\mu + \frac{\sigma^2 \mu^2}{2}}\right) .
\]
Since $\M_P[\mu] \geq e^{m\mu}$, conditions \eqref{ass:mu1}, \eqref{ass:mu2} can only be satisfied if $m\leq \ln\sqrt{2}$.
Obviously, as $m\to -\infty$, the critical value of $\sigma$ tends to $+\infty.$
We resolved the conditions \eqref{ass:mu1}, \eqref{ass:mu2} numerically for $\mu\in [-\ln\sqrt{2}, \ln\sqrt{2}]$, using the matlab procedure {\tt fminbnd}.
The result is shown in Fig. \ref{fig:truncGauss}, where we plot the critical value of $\sigma$
as a function of the parameter $m$.
%for $m\in [-e^{-1},e^{-1}]$.

\begin{figure}
\centerline{
%\hskip 3mm
\includegraphics[width=0.6\columnwidth]{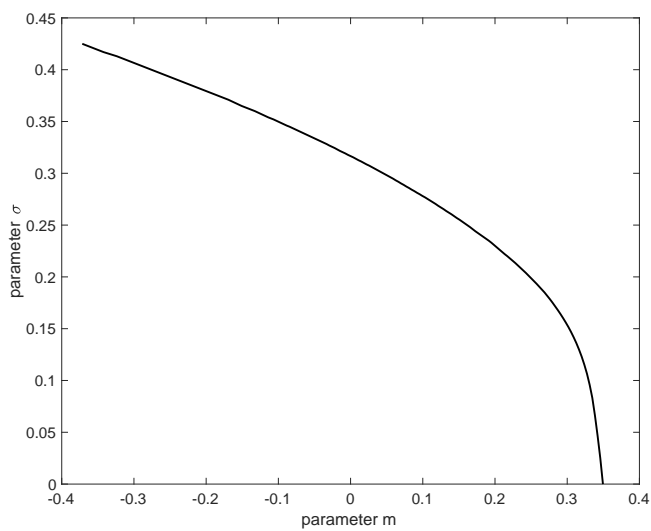}}
\caption{Numerical calculation of the critical value of the parameter $\sigma$
as a function of the parameter $m$ for the truncated normal distribution \eqref{truncN}.
}
\label{fig:truncGauss}
\end{figure}

\section*{Acknowledgment}
JH acknowledges the support of the KAUST baseline funds.
This work was done partially while the author was visiting the Institute for Mathematical Sciences, National University of Singapore in 2019.
The visit was supported by the Institute.

\vspace{7mm}

\end{document}